\documentclass[reqno,T1]{amsproc}
\usepackage{amssymb}
\usepackage{amsmath}
\usepackage{amsfonts}
\usepackage[dvipsnames]{xcolor}
\usepackage{geometry}
\usepackage{xifthen} % provides \isempty
\usepackage{xparse} % new commands with more than one optional argument
\usepackage{bbm} % for \mathbbm{1}
\usepackage{mathtools}

\usepackage{tikz}
\usetikzlibrary{arrows,cd,decorations.markings,arrows.meta}
\newcommand{\monic}{>->}	% hookrightarrow would also be an option
\newcommand{\epic}{->>}

\definecolor{myurlcolor}{rgb}{0,0,0.4}
\definecolor{mycitecolor}{rgb}{0,0.5,0}
\definecolor{myrefcolor}{rgb}{0.5,0,0}
\usepackage[pagebackref,draft=false]{hyperref}
\hypersetup{colorlinks,
linkcolor=myrefcolor,
citecolor=mycitecolor,
urlcolor=myurlcolor}

\usepackage{cleveref}

\newcommand{\beq}{\begin{equation}}
\newcommand{\eeq}{\end{equation}}
\newcommand{\Z}{\mathbb{Z}}

\newcommand{\R}{\mathbb{R}}
\newcommand{\op}{\mathrm{op}}
\newcommand{\eps}{\varepsilon}

\newcommand{\newterm}[1]{\textbf{#1}}
\newcommand{\conv}[1]{\mathrm{conv}(#1)}

\newcommand{\cat}{\mathsf{C}}
\newcommand{\Norm}{\mathsf{Norm}}
\newcommand{\ideal}{\mathcal{N}}	% ideal of small/negligible morphisms
\newcommand{\iker}[1]{\mathrm{ker}(#1)}		% kernel wrt ideal
\newcommand{\icoker}[1]{\mathrm{coker}(#1)}	% cokernel wrt ideal
\newcommand{\iim}[1]{\mathrm{im}(#1)}		% normal image wrt ideal
\newcommand{\icoim}[1]{\mathrm{coim}(#1)}	% normal coimage wrt ideal
\newcommand{\iKer}[1]{\mathrm{Ker}(#1)}		% kernel wrt ideal object
\newcommand{\iCoker}[1]{\mathrm{Coker}(#1)}	% cokernel wrt ideal object
\newcommand{\iIm}[1]{\mathrm{Im}(#1)}		% normal image wrt ideal object
\newcommand{\iCoim}[1]{\mathrm{Coim}(#1)}	% normal coimage wrt ideal object
\newcommand{\arrow}[1]{\mathsf{Arr}(#1)}
\newcommand{\id}{\mathrm{id}}
\newcommand{\Nsb}[1]{\mathsf{nSub}(#1)}		% normal subobjects
\newcommand{\Nqt}[1]{\mathsf{nQot}(#1)}		% normal subobjects
\newcommand{\Pos}{\mathsf{Poset}}

\newcommand{\dom}[1]{\mathrm{dom}(#1)}		% domain
\swapnumbers
\newtheorem{dummy}{Dummy}[section]
\newtheorem{thm}[dummy]{Theorem}
\newtheorem{ass}[dummy]{Assumption}
\newtheorem{lem}[dummy]{Lemma}
\newtheorem{prop}[dummy]{Proposition}
\newtheorem{cor}[dummy]{Corollary}

\newtheorem{qstn}[dummy]{Question}
\newtheorem{defn}[dummy]{Definition}

\makeatletter
\newtheorem*{rep@theorem}{\rep@title}
\newcommand{\newreptheorem}[2]{%
\newenvironment{rep#1}[1]{%
 \def\rep@title{#2 \ref{##1}}%
 \begin{rep@theorem}}%
 {\end{rep@theorem}}}
\makeatother
\newreptheorem{thm}{Theorem}

\theoremstyle{remark}
\newtheorem{ex}[dummy]{Example}
\newtheorem{rem}[dummy]{Remark}

\numberwithin{equation}{section}

\allowdisplaybreaks

\usepackage{todonotes}

\begin{document}
\sloppy

% vertical spacing in multiline equations
\setlength{\jot}{6pt}

%-------------------------------------------------------------------

%%%%%%%%%%%% title page stuff %%%%%%%%%%%%%%%%%%%%%%%%%%

\title{Non-abelian and $\eps$-curved homological algebra with arrow categories}

\author{Tobias Fritz}

\address{Department of Mathematics, University of Innsbruck, Austria}
\email{tobias.fritz@uibk.ac.at}



\thanks{\textit{Acknowledgements.} We thank William Slofstra for many interesting discussions and Marco Grandis as well as an anonymous referee for helpful feedback on the paper.}

\begin{abstract}
	Grandis's non-abelian homological algebra generalizes standard homological algebra in abelian categories to \emph{homological categories}, which are a broader class of categories including for example the category of lattices and Galois connections.
	Here, we prove that if $\cat$ is any category with an ideal of null morphisms with respect to which (co)kernels exist, then the arrow category of $\cat$ is a homological category.
	This broadens the applicability	of Grandis's framework substantially.
	In particular, one can form the homology of chain complexes in $\cat$ by taking the homology objects to be morphisms of $\cat$, which one may think of as maps from an object of cycles to an object of chains modulo boundaries.

	One situation to which Grandis's original framework does not apply directly is \emph{$\eps$-curved homological algebra}.
	This refers to chain complexes of normed spaces whose differential squares to zero only approximately, in the sense that $\|d^2\| \leq \eps$ for some $\eps > 0$. This is relevant for example in the theory of approximate representations of groups, where Kazhdan has successfully employed $\eps$-curved homological techniques in an ad-hoc manner.
	We develop some basics of $\eps$-curved homological algebra and note that our result on arrow categories facilitates the application of Grandis's theory.
\end{abstract}

\maketitle

\tableofcontents

\section{Introduction}

The field of non-abelian homological algebra attempts to generalize the methods of homological algebra to categories that are not necessarily abelian~\cite{bb,grandis}.
One of the motivating examples is the long exact sequence of homotopy groups, which generically involves nonabelian groups and is therefore not covered by homological algebra in an abelian category.

A powerful categorical framework for non-abelian homological algebra has been developed in Marco Grandis's book \emph{Homological Algebra in Strongly Non-Abelian Settings}~\cite{grandis}, based on work going back to 1992~\cite{grandis_original}.
Many of the stronger results of this theory apply to \newterm{homological categories} in Grandis's sense, which are categories equipped with an ideal of morphisms called \newterm{null morphisms} and satisfying a bunch of axioms.
Unfortunately these axioms are still somewhat restrictive: for example, the category of groups is not homological.
A known workaround, which has been used extensively by Grandis, is to work with \emph{categories of pairs} instead, such as the category of pairs of groups.

\medskip

In this paper, we first redevelop a few basic results on kernels and cokernels in categories with null morphisms and provide a short recap on homological categories (\Cref{background}).
We then prove our first main result:

\begin{repthm}{Carrow}
	If $\cat$ is any category with null morphisms in which kernels and cokernels exist, then its arrow category acquires the structure of a homological category.
\end{repthm}

Since arrow categories are closely related to categories of pairs (the latter being a full subcategory of the former), one may think of this result as providing a general justification for the idea of moving to categories of pairs, but extended to the whole arrow category.

We then develop some further basic aspects of how to do homological algebra with such arrow categories (rest of \Cref{sec_arrows}).
In particular, given a chain complex in $\cat$, there is a well-behaved notion of its homology as a sequence of objects in $\arrow{\cat}$, and there is a long exact sequence in $\arrow{\cat}$ associated to every short exact sequence of chain complexes in $\cat$. 
This broadens the applicability of Grandis's framework.
For example the category of groups is now covered, since its arrow category is a well-behaved environment for homological algebra by virtue of being a homological category.

\medskip

In \Cref{seceps}, we turn to \newterm{$\eps$-curved homological algebra}. By this we mean homological algebra with normed spaces for which the differential $d$ squares to zero only approximately\footnote{The phenomenon that a differential does not need to square to zero is also known as \emph{curvature}, since it corresponds to curvature in differential geometry; see e.g.~\cite[(A.43)]{gasperini}.}, meaning that $\|d^2\| \le \eps$ in operator norm for a fixed parameter $\eps \in (0,1)$.
We show how this fits into our framework by starting with the category of normed spaces and declaring a morphism to be null if it is of operator norm $\le \eps$.
We construct kernels and cokernels in this category (\Cref{Normideal}) and a number of related technical results for it.

Our main motivation for this work comes from Kazhdan's results on approximate representations of groups~\cite{kazhdan}. Among other things, what he proved is that every approximate unitary representation of an amenable group can be perturbed to an actual representation. His method of proof has a clear cohomological flavour, despite not being cohomological in the strict sense: using coefficients in an approximate representation of a group does not make the group cohomology differential square to zero, but merely makes it square have small norm.
We briefly indicate towards the end how our results match Kazhdan's definition of $\eps$-acyclicity (\Cref{kazhdan}).

\medskip

We hope that the results of this paper point in a fruitful direction towards a framework for homological algebra in which Kazhdan's results find their conceptual underpinnings and can be taken further.
Another potential application may be to Shalom's characterization of property (T) in terms of vanishing first reduced cohomology~\cite{shalom}, for which one may hope for a quantitative version constructed in terms of our framework.\footnote{We thank Andreas Thom for suggesting this idea.}
Other approaches for developing such theories are conceivable, such as equipping curved differential algebras and curved $A_\infty$-algebras~\cite{brzezinski,CT} with suitably compatible norms.
But these are not the topic of this paper.

\subsection*{Notation}

In our diagrams, we use bullets ``$\bullet$'' to denote unnamed objects. We denote normal monomorphisms by $\!\!\!\!\begin{tikzcd} \ar[\monic]{r} & {} \end{tikzcd}\!\!\!\!$ and normal epimorphisms by $\!\!\!\!\begin{tikzcd} \ar[\epic]{r} & {} \end{tikzcd}\!\!\!\!$. Dotted arrows $\!\!\!\!\begin{tikzcd} \ar[densely dotted]{r} & {} \end{tikzcd}\!\!\!\!$ denote null morphisms.

\section{Categories with null morphisms and homological categories}
\label{background}

Here, we develop some basic theory of kernels and cokernels for categories with null morphisms.
Although this does not go beyond the results obtained by Grandis, this serves the purpose of developing some of this theory with minimal assumptions and in addition makes the paper accessible without any knowledge of Grandis's book~\cite{grandis}.

Throughout, we assume that $\cat$ is any category and $\ideal \subseteq \cat$ is any ideal of morphisms called \newterm{null morphisms}. This gives us notions of \emph{kernel} and \emph{cokernel}, both with respect to $\ideal$:

\begin{defn}
	For $f : A \to B$ be a morphism in $\cat$, a \newterm{kernel} is a morphism $\iker{f} : \iKer{f} \to A$ such that $f\circ\iker{A}$ is null, and such that every $g$ which makes $fg$ null uniquely factors through the kernel,
			
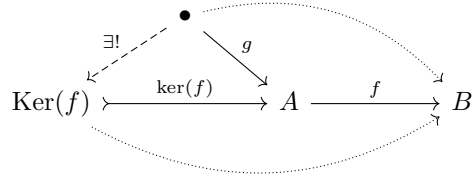
\[
		\begin{tikzcd}
			& \bullet \ar["g"]{dr} \ar[densely dotted,bend left]{rrrd} \ar[dashed,"\exists!",swap]{dl} \\
			\iKer{f} \ar[\monic,"\iker{f}"]{rr} \ar[densely dotted,bend right]{rrrr} & & A \ar["f"]{rr} & & B
		\end{tikzcd}
	\]
\end{defn}

Here, the dotted arrows denote the null composites. An immediate consequence is that the kernel is characterized up to unique isomorphism in the slice category over $A$.
A \newterm{cokernel}
\[
	\icoker{f} : B \to \iCoker{f}
\]
is defined dually, and it is similarly unique in the coslice under $B$.
Since naming the kernel and cokernel objects is usually not necessary, we often merely write ``$\bullet$'' in place of the object names $\iKer{f}$ and $\iCoker{f}$, labelling only the morphisms.

\begin{ass}
	Throughout this section, we work in a category $\cat$ with an ideal of null morphisms $\ideal$ such that $\cat$ has kernels and cokernels (with respect to $\ideal$).
\end{ass}

\begin{rem}
	Grandis's results in~\cite{grandis} are formulated under the additional assumption that $\ideal$ is ``closed'', by which he means that every null morphisms factors through a null identity morphisms.
	However, the proofs of his basic results do not make use of this property, and it is also an assumption that we cannot make, since it is not satisfied in the category presented in \Cref{seceps}. This is why we redevelop the basic theory here with our own proofs, ensuring that closedness of the ideal is not needed.
\end{rem}

It is easy to see that every kernel is a monomorphism and every cokernel an epimorphism. The \newterm{image}
\[
	\iim{f} : \iIm{f} \to B
\]
of a morphism $f : A \to B$ is the kernel of its cokernel, and its \newterm{coimage}
\[
	\icoim{f} : A \to \iCoim{f}
\]
is the cokernel of its kernel.\footnote{In both cases, we drop the adjective ``normal'' used in~\cite{grandis}.} The cokernel of the image of a morphism is its mere cokernel, and similarly the kernel of the coimage is the kernel.

\begin{ex}
	\label{groups}
	In the category of groups with null morphisms being the trivial homomorphisms, kernels and cokernels exist. For a homomorphism $f : G \to H$, we have that $\icoim{f}$ is the usual image, while $\iim{f}$ is the normal subgroup generated by the usual image.
\end{ex}

A \newterm{normal monomorphism} is any morphism that is a kernel, and likewise a \newterm{normal epimorphism} is any cokernel. Given an object $A$, the isomorphism classes\footnote{We abuse notation and terminology by leaving the distinction between the equivalence class and a representative mostly implicit.} of normal monomorphisms into $A$ form a partially ordered collection of \newterm{normal subobjects} $\Nsb{A}$, which may or may not be a (small) set.
Similarly, the isomorphism classes of normal epimorphisms out of $A$ form a collection of \newterm{normal quotient objects} $\Nqt{A}$. 
It is straightforward to see that every normal monomorphism is the kernel of its cokernel, and similarly every normal epimorphism is the cokernel of its kernel.
Therefore taking kernels and cokernels implements a bijection
\beq
	\label{nqt_nsb}
	\Nqt{A} \cong \Nsb{A}.
\eeq
We consider $\Nqt{A}$ as partially ordered in such a way that this bijection is order-preserving, or equivalently such that $s \le t$ for $s, t \in \Nqt{A}$ if and only if $s$ factors across $t$.\footnote{Note that Grandis in~\cite{grandis} uses the opposite convention. Ours has the advantage that $s \le t$ for normal quotient objects is equivalent to $\iker{s} \le \iker{t}$ for normal subobjects, making the identification between normal subobjects and normal quotients more intuitive. On the flip side, taking the categorical dual of a statement now also requires reversing $\le$ in addition to the direction of the morphisms.}
Finally, a \newterm{short exact sequence} is a diagram
\[
	\begin{tikzcd}
		\bullet \ar["s",\monic]{r} & A \ar["t",\epic]{r} & \bullet
	\end{tikzcd}
\]
such that $s \in \Nsb{A}$ and $t \in \Nqt{A}$ correspond under the bijection, meaning that $s = \iker{t}$ and $t = \icoker{s}$.

\begin{ex}
	In the category of groups as in \Cref{groups}, the normal monomorphisms are the injective homomorphisms with normal image, and the normal epimorphisms are the surjective homomorphisms.
	In particular, the composition of two normal monomorphisms does not need to be normal.
	$\Nsb{G}$ for a group $G$ is the lattice of normal subgroups.
	Short exact sequences are the usual ones.
\end{ex}

We will see in \Cref{lattice} that $\Nsb{A}$ is a bounded lattice. The existence of a least element $\iker{\id}\in\Nsb{A}$ and a greatest element $\id = \iim{\id}\in\Nsb{A}$ are already obvious.

\begin{lem}
	Let $f$ and $g$ be composable.
	\begin{enumerate}
		\item\label{23mono} If $g$ is a monomorphism and $gf$ a normal monomorphism, then so is $f$.
		\item\label{23epi} If $f$ is an epimorphism and $gf$ a normal epimorphism, then so is $g$. 
	\end{enumerate}
\end{lem}

\begin{proof}
	By duality it is enough to prove the first statement. By assumption, $gf$ is the kernel of its cokernel; we prove that the same holds for $f$, considering the diagram
	
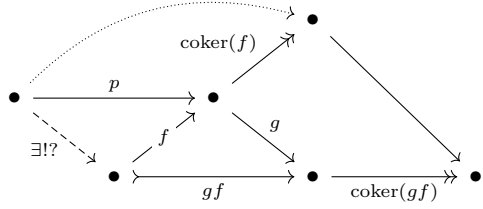
\[
		\begin{tikzcd}
			& & & \bullet \ar{ddrr} \\
			\bullet \ar["\exists !?",swap,dashed]{dr} \ar["p"]{rr} \ar[densely dotted,bend left]{rrru} & & \bullet \ar["g"]{rd} \ar["\icoker{f}",\epic]{ur} \\
			& \bullet \ar["f" description]{ur} \ar[swap,"gf",\monic]{rr} & & \bullet \ar[swap,\epic,"\icoker{gf}"]{rr} & & \bullet 
		\end{tikzcd}
	\]
	Here, the extra diagonal morphism on the right is induced by the universal property of $\icoker{f}$, and $p$ is any morphism such that $\icoker{f}\circ p$ is null. Since $f$ is guaranteed to be a monomorphism by a standard property of monomorphisms, it is enough to show that $p$ factors through $f$ as indicated by the dashed arrow. Since $\icoker{gf}\circ gp$ is null by the commutativity of the diagram, the assumption that $gf$ is the kernel of its cokernel implies that we get the desired dashed arrow.
	To see that its composition with $f$ recovers $p$, we use that $g$ is a monomorphism together with the construction of the dashed arrow as a factorization of $gp$ across $gf$.
\end{proof}

\begin{ex}
	To see that \ref{23mono} is false without the assumption that $g$ is a monomorphism, take $\cat$ to be the category of groups as in \Cref{groups}, let $f$ be an inclusion of a two-element subgroup $\Z_2 \subseteq S_3$, and take $g$ to be the quotient homomorphism $S_3 \to \Z_2$.
	Then $gf$ is a normal monomorphism, while $f$ is a non-normal monomorphism.
\end{ex}

We now consider how normal subobjects transfer along morphisms. 

\begin{defn}
	\label{pushpull}
	Let $f : A \to B$. Then every $s\in\Nsb{B}$ has a \newterm{pullback}
	\[
		f^*s \coloneqq \iker{\icoker{s}f}
	\]
	and every $t\in\Nqt{A}$ has a \newterm{pushforward}
	\[
		f_*t \coloneqq \icoker{f\,\iker{t}}.
	\]
\end{defn}

Through the bijection $\Nsb{A} \cong \Nqt{A}$, we can also apply the pullback operation to normal quotient objects and the pushforward operation to normal subobjects.

\begin{lem}
	\label{pullpush}
	In the situation of \Cref{pushpull}, the diagram
	\[
		\begin{tikzcd}
			\bullet \ar[\monic,"f^*s",swap]{d} \ar[\monic]{r} & \bullet \ar[\monic,"s"]{d} \\
			A \ar["f"]{r} & B
		\end{tikzcd}
	\]
	is a pullback in $\cat$, and
	\[
		\begin{tikzcd}
			A \ar["f"]{r} \ar["t",swap,\epic]{d} & B \ar[\epic,"f_*t"]{d} \\
			\bullet \ar[\epic]{r} & \bullet
		\end{tikzcd}
	\]
	is a pushout in $\cat$.
\end{lem}

In particular, pullbacks of normal subobjects can be computed as ordinary pullbacks in $\cat$, and similarly pushforwards as pushouts.
This follows without any assumption on the existence of ordinary limits or colimits in $\cat$.

\begin{proof}
	By duality it is enough to prove the first statement.
	We do so using the diagram
	
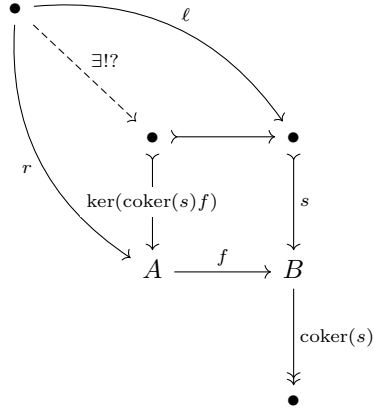
\[
		\begin{tikzcd}[sep=1.3cm]
			\bullet \ar[bend left,"\ell"]{drr} \ar[bend right,"r",swap]{ddr} \ar[dashed,"\exists!?"]{dr} \\
			& \bullet \ar[\monic,"\iker{\icoker{s}f}" description,swap]{d} \ar[\monic]{r} & \bullet \ar[\monic,"s"]{d} \\
			& A \ar["f"]{r} & B \ar[\epic,"\icoker{s}"]{d} \\
			& & \bullet
		\end{tikzcd}
	\]
	for given $\ell$ and $r$ that satisfy $s\ell = fr$.
	Since $f^*s = \iker{\icoker{t}f}$ is a monomorphism, it is enough to show that a dashed arrow exists which makes the diagram commute. Since $\icoker{s}s$ is null, we therefore know that $\icoker{s}fr$ is null as well. Hence $r$ factors through the specified kernel, giving the dashed arrow making the $r$-triangle commute. Now the $\ell$-triangle commutes automatically since $s$ is a monomorphism.
\end{proof}

Although normal monomorphisms do not need to be closed under composition\footnote{The non-transitivity of normal subgroup inclusion is the paradigmatic example.}, this does hold in certain situations like the following.

\begin{lem}
	\label{doublepull}
	Let $s,t \in\Nsb{A}$. Then $s \circ s^*t$ and $t \circ t^*s$ are normal monomorphisms representing the same normal subobject of $A$, namely $s_* s^* t = t_* t^* s$.
\end{lem}

\begin{proof}
	Consider the diagram
	
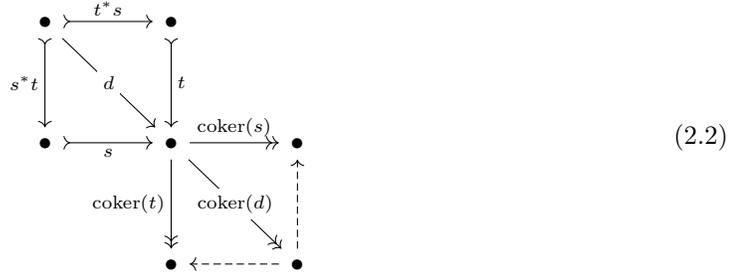
\begin{equation}
		\label{doublepulldiag}
		\begin{tikzcd}[sep=1.2cm]
			\bullet \ar["d" description]{rd} \ar[\monic,"t^*s"]{r} \ar[swap,\monic,"s^*t"]{d} & \bullet \ar[\monic,"t"]{d} \\
			\bullet \ar[\monic,swap,"s"]{r} & \bullet \ar[\epic,"\icoker{s}"]{r} \ar[\epic,swap,"\icoker{t}"]{d} \ar[\epic,"\icoker{d}" description]{dr} & \bullet \\
			& \bullet & \bullet \ar[dashed]{l} \ar[dashed]{u}
		\end{tikzcd}
	\end{equation}
	where the diagonal $d$ is defined by the commutativity, and where we have chosen representing morphisms for the normal subobjects $s^* t$ and $t^* s$ in such a way that they have the same domain and the upper square commutes; this is possible by \Cref{pushpull}.
	We show first that $d$ is a normal monomorphism.
	The dashed arrows are induced by the universal property of $\icoker{d}$. Using the fact that $s$ and $t$ are the kernels of their respective cokernels, as well as the fact that the upper square is a pullback, it is now straightforward to check that $d$ is also the kernel of its cokernel, and therefore a normal monomorphism.

	This also implies that $s_* s^* t$ is represented by $d$, as is $t_* t^* s$ by symmetry. Hence we obtain $s_* s^* t = t_* t^* s$.
\end{proof}

\begin{lem}[{\cite[1.5.8]{grandis}}]
	\label{galois_conn}
	For $f : A \to B$ and $s \in \Nsb{A}$ and $t \in \Nsb{B}$, 
	we have
	\[
		s \leq f^* t \textrm{ in } \Nsb{A} \qquad\Longleftrightarrow\qquad f_* s \leq t \textrm{ in } \Nsb{B}.
	\]
\end{lem}

Hence we have a Galois connection $f_* \dashv f^*$, and the category of elements of the functor $\Nsb{-} : \cat \to \Pos$ is a Grothendieck bifibration over $\cat$.

\begin{proof}
	We have a diagram
	\[
		\begin{tikzcd}[sep=1cm]
			\bullet \ar[\monic,"s"]{dr} \ar[densely dotted,bend left]{rr} \ar[dashed,bend right]{dd} & & \bullet & & \bullet \ar["\iker{\icoker{fs}}" description,\monic]{dl} \ar[densely dotted,bend right]{ll} \ar[dashed,bend left]{dd} \\
			& A \ar["f"]{rr} & & B \ar[\epic,"\icoker{fs}" description]{ul} \ar[\epic,"\icoker{t}" description]{dl} \\
			\bullet \ar[swap,\monic,"\iker{\icoker{t}f}" description]{ur} \ar[densely dotted,bend right]{rr} & & \bullet \ar[dashed,bend left,crossing over,<-]{uu} & & \bullet \ar[\monic,"t",swap]{ul} \ar[densely dotted,bend left]{ll}
		\end{tikzcd}
	\]
	where the dotted composites are null. Applying the universal properties shows that the dashed arrow on the left exists if and only if the one in the middle does if and only if the one on the right does, using also the fact that $t = \iker{\icoker{t}}$ and similarly for $fs$.
\end{proof}

\begin{lem}
	\label{fstar}
	For
	\[
		\begin{tikzcd}
			A \ar["f"]{r} & B \ar["g"]{r} & C
		\end{tikzcd}
	\]
	in $\cat$, we have
	\begin{align}
		\iker{gf} & = f^* \iker{g}, & \icoker{gf} & = g_* \icoker{f}, \\
		g_* \iim{f} & = \iim{gf}, & f^* \icoim{g} & = \icoim{gf}.
	\end{align}
\end{lem}

\begin{proof}
	By duality, it is enough to prove the first statement in each row.
	For the statement on kernels, consider the diagram
	\[
		\begin{tikzcd}
			\bullet \ar["\iker{\icoim{g}f}",\monic]{drr} & & & \bullet \ar[\monic,"\iker{g}"]{dr} \ar[bend left,densely dotted]{drrr} \\
			& & A \ar["f"]{rr} & & B \ar[swap,"\icoim{g}",\epic]{dr} \ar["g"]{rr} & & C \\
			\bullet \ar[\monic,"\iker{gf}",swap]{urr} & & & & & \bullet \ar{ur}
		\end{tikzcd}
	\]
	where the dotted arrow is null.
	Our goal is to show that the two kernels on the left are isomorphic.
	To this end, it is enough to show that if $p$ is any other morphism with codomain $A$, then $gfp$ is null if and only if $\icoim{g}fp$ is null. Since the ``if'' direction is trivial, we focus on the ``only if''. So assume that $gfp$ is null. Then $fp$ factors through $\iker{g}$; but this implies that also $\icoim{g}fp$ is null, since already $\icoim{g}\iker{g}$ is.

	For the statement on images, consider the diagram
	\[
		\begin{tikzcd}
			& \bullet \ar["\iim{f}",\monic]{dr} & & & & X \\
			A \ar{ur} \ar["f"]{rr} & & B \ar[\epic,swap,"\icoker{f}"]{dr} \ar["g"]{rr} & & C \ar["q"]{ur} \\
			& & & \bullet \ar{ur}
		\end{tikzcd}
	\]
	In order to prove the claim, it is enough to show that $\icoker{gf}$ coincides with $\icoker{g\,\iim{f}}$. In other words, the composite of any $q$ out of $C$ with $gf$ is null if and only if the composite with $g\,\iim{f}$ is null. Again the ``if'' part is trivial, so we focus on the ``only if''. Assuming that $qgf$ is null, we can factor $qg$ through $\icoker{f}$; but then clearly also $qg\,\iim{f}$ is null, since $\icoker{f}\iim{f}$ is already null.
\end{proof}

For $f : A \to B$, we have a monotone map on normal subobjects given by
\[
	f^* f_* \: : \: \Nsb{A} \longrightarrow \Nsb{A}.
\]
By \Cref{galois_conn} and a standard properties of Galois connections, $f^* f_*$ is a closure operator on $\Nsb{A}$.

\begin{lem}
	\label{closed_vs_ker}
	For $f : A \to B$, the $f^* f_*$-closed normal subobjects of $A$ are precisely those of the form $\iker{gf}$ for some $g$ out of $B$. 
\end{lem}

\begin{proof}
	By definition of $f^*$, every $f^* f_*$-closed $k \in \Nsb{A}$ is of the form $\iker{gf}$, namely with $g = \icoker{f_*k}$. Conversely, we need to show that every $\iker{gf} \in \Nsb{A}$ is $f^* f_*$-closed. This follows from two applications of \Cref{fstar} and a standard property of Galois connections,
	\[
		f^* f_* \iker{gf} = f^* f_* f^* \iker{g} = f^* \iker{g} = \iker{fg},
	\]
	as was to be shown.
\end{proof}

\begin{lem}[{\cite[1.5.8]{grandis}}]
	\label{lattice}
	Every $\Nsb{A}$ is a bounded lattice.
\end{lem}

\begin{proof}
	Having noted the existence of least and greatest elements above, it remains to explain how to construct meets and joins in $\Nsb{A}$. So let $s,t\in\Nsb{A}$.
	Starting with meets, we claim that the meet of $s$ and $t$ is $s_* s^* t = t_* t^* s$ (\Cref{doublepull}). Indeed this claim follows from the upper square in \eqref{doublepulldiag} being a pullback.

	The construction of the join $s\lor t$ is dual, and follows from the bijection between normal subobjects and normal quotient objects, together with the existence of meets of normal subobjects in $(\cat^\op,\ideal^\op)$. Concretely, this means that the join can be constructed as
	\beq
		\label{join_formula}
		\icoker{t}^* \icoker{t}_* s \,=\, \icoker{s}^* \icoker{s}_* t,
	\eeq
	establishing the claim.
\end{proof}

Let us now turn to homology.

\begin{defn}
	\label{chaincomplex}
	\begin{enumerate}
		\item A \newterm{chain complex} $C = (C_\bullet,d_\bullet)$ is a sequence of objects $C_\bullet = (C_n)_{n \in \Z}$ together with a sequence of morphisms $d_\bullet = (d_n : C_n \to C_{n-1})_{n \in \Z}$ such that $d_{n-1} d_n \in \ideal$ for every $n\in\Z$.
		\item Such a chain complex is \newterm{exact} if $\iim{d_n} = \iker{d_{n-1}}$ in $\Nsb{C_{n-1}}$, or equivalently $\icoim{d_{n-1}} = \icoker{d_n}$ in $\Nqt{C_{n-1}}$, for all $n \in \Z$.
	\end{enumerate}
\end{defn}

Chain complexes form a category as usual, in which the morphisms $(C_\bullet,d_\bullet) \to (D_\bullet,d_\bullet)$ are the sequences of morphisms $f_\bullet = (f_n : C_n \to D_n)_{n\in\Z}$ commuting with the differentials,
\[
	d_n f_n = f_{n-1} d_n.
\]
It is straightforward to show that this makes chain complexes into a category.
Moreover, with null morphisms those $f_\bullet$ of which all components are null, the category of chain complexes also has kernels and cokernels which can be computed degreewise.
A short exact sequence of chain complexes is therefore a diagram
\[
	(C_\bullet,d_\bullet)\longrightarrow (D_\bullet,d_\bullet) \longrightarrow (E_\bullet,d_\bullet)
\]
consisting of two morphisms of chain complexes such that every component $C_n \to D_n \to E_n$ is short exact in $\cat$.

The homology objects of a chain complex should be certain subquotients of the objects in the complex formed out of the kernels and images of the differentials.
In order to allow for a well-behaved theory of homology with many of the usual properties known from homological algebra in abelian categories, Grandis has imposed extra conditions as follows.

\begin{defn}[{\cite[1.3.6]{grandis}}]
	\label{homologicalcatdef}
	The pair $(\cat,\ideal)$ is a \newterm{homological category} if the following hold:
	\begin{enumerate}
		\item\label{coker_exist} Kernels and cokernels exist.
		\item\label{null_exist} Every null morphism $f : A \to B$ factors through a null identity morphism.
		\item Normal monomorphisms and normal epimorphisms are closed under composition.
		\item\label{homologyaxiom} The \newterm{homology axiom}: given a normal monomorphism $i$ and a normal epimorphism $q$ as in
			\[
				\begin{tikzcd}
					A \ar[\monic,"i"]{r} \ar[\epic,dashed]{d} & B \ar[\epic,"q"]{d} \\
					H \ar[\monic,dashed]{r} & D
				\end{tikzcd}
			\]
			and such that $\icoker{i} \, \iker{q} \in \ideal$,
			there exists an object $H$ and a normal monomorphism and normal epimorphism, as indicated by dashed arrows, such that the square commutes.
	\end{enumerate}
\end{defn}

For example, the category of lattices as objects and Galois connections as morphisms is homological. For this and other examples, see~\cite[Sections~1.4~and~1.6]{grandis}.
Note also that this notion of homological category is to be distinguished from the one of Borceux and Bourn~\cite[Chapter~4]{bb}.
In contrast to the latter, a homological category is not determined by its underlying category alone, since various choices of null morphisms can result in distinct homological categories.
The freedom of choice of null morphisms will play an important role for us in \Cref{seceps}.

In order to have more convenient terminology for~\ref{null_exist}, let us say that a \newterm{null object} is an object whose identity morphism is null. 
The assumption is then that the null morphisms are exactly those morphisms that factors across a null object.
It follows that a morphism $f$ is null if and only if $\iIm{f}$, or equivalently $\iCoim{f}$, is a null object.

Based on~\ref{coker_exist} and~\ref{null_exist} only, Grandis showed in~\cite[1.5.5]{grandis} that every morphism $f : A \to B$ has a factorization
\beq
	\label{imcoim_factor}
	\begin{tikzcd}[column sep=35pt]
		A \ar[\epic,"\icoim{f}"]{r} & \iCoim{f} \ar{r} & \iIm{f} \ar[\monic,"\iim{f}"]{r} & B
	\end{tikzcd}
\eeq
called the \newterm{normal factorization}. The homology axiom \ref{homologyaxiom} can now also be stated as the requirement that, whenever $i$ and $q$ are as assumed, then the normal factorization of $qi$ is such that the middle morphism is an isomorphism. In particular, the homology object $H$ is both the image and the coimage of $qi$, and it be constructed in terms of pushforward and pullback as
\[
	\begin{tikzcd}
		A \ar[\monic,"i"]{r} \ar[\epic,"i^* q"']{d} & B \ar[\epic,"q"]{d} \\
		H \ar[\monic,"q_* i"']{r} & D
	\end{tikzcd}
\]
This shows in particular that the homology object is unique up to unique isomorphism.
And while the condition $\icoker{i} \, \iker{q} \in \ideal$ in~\ref{homologyaxiom} is nicely self-dual, an arguably more intuitive form is $\iker{q} \le i$, or equivalently $q \le \icoker{i}$, corresponding to the usual inclusion of boundaries in cycles if one thinks of $q$ as taking the quotient by boundaries.\footnote{Or yet equivalently $q \le \icoker{i}$.}

Given a chain complex $C = (C_\bullet, d_\bullet)$ in a homological category, we can therefore form the \newterm{homology object} at every $H_n(C)$ at $n \in \Z$ by applying \Cref{homologicalcatdef}\ref{homologyaxiom} with $q \coloneqq \icoker{d_{n-1}}$ and $i \coloneqq \iker{d_n}$.
Hence to every chain complex $C$ is associated a sequence of homology objects $H_\bullet(C)$.

One of the most basic and central themes of homological algebra is the development of long exact sequences for short exact sequences of chain complexes.
We will restate Grandis's main result on this in order to give a flavour of what his theory achieves, but this requires a bit more preparation.
Throughout the following, $(\cat,\ideal)$ is a homological category.

\begin{defn}[{\cite[2.3.1]{grandis}}]
	A morphism $f : A\to B$ is
	\begin{enumerate}
		\item \newterm{left modular} if for every $s\in\Nsb{A}$,
			\beq
				\label{left_modular}
				f^* f_* s = s \lor \iker{f}.
			\eeq
		\item \newterm{right modular} if for every $t\in\Nsb{B}$,
			\beq
				\label{right_modular}
				f_* f^* t = t \land \iim{f}.
			\eeq
	\end{enumerate}
\end{defn}

By the construction of meets and joins in the proof of \Cref{lattice}, every normal monomorphism is automatically right modular, and every normal epimorphism is automatically left modular.

\begin{defn}
	\label{modularitydef}
	$(\cat,\ideal)$ satisfies the \newterm{modularity condition} if the following hold:
	\begin{enumerate}
		\item For every object $A$, the lattice $\Nsb{A}$ is modular, meaning that
			\[
				s_1 \leq s_2 \quad \Longrightarrow \quad s_1 \lor (t \land s_2) = (s_1 \lor t) \land s_2
			\]
			for all $s_1,s_2,t\in\Nsb{A}$.
		\item Every morphism is left and right modular.
	\end{enumerate}
\end{defn}

\noindent
As explained in~\cite[1.2.8]{grandis}, the modularity condition implies Frobenius reciprocity in the form
\begin{align*}
	f^*(f_*s \lor t) & = s \lor f^* t, \\
	f_*(f^*t \land s) & = t \land f_*s
\end{align*}
for all morphisms $f : A \to B$ and $s \in \Nsb{A}$ and $t \in \Nsb{B}$.

\begin{thm}[{\cite[3.3.5]{grandis}}]
	\label{longexact}
	Let $(\cat,\ideal)$ be a homological category. Let
	\[
		(C_\bullet,d_\bullet)\longrightarrow (D_\bullet,d_\bullet) \longrightarrow (E_\bullet,d_\bullet)
	\]
	be a short exact sequence of chain complexes. Then there is an induced long sequence of homology objects
	\[
		\begin{tikzcd}[column sep=19pt]
			\ldots \ar{r} & H_{n+1}(E) \ar["\partial_{n+1}"]{r} & H_n(C) \ar{r} & H_n(D) \ar{r} & H_n(E) \ar["\partial_n"]{r} & H_{n-1}(C) \ar{r} & \ldots
		\end{tikzcd}
	\]
	such that:	
	\begin{enumerate}
		\item The unlabelled arrows are induced by functoriality of homology.
		\item The connecting morphisms $\partial_n$ are natural in the short exact sequence.
		\item Every composite of two morphisms in the sequence is null.
		\item The sequence is exact if the following instances of the modularity condition hold in $\Nsb{D_n}$ for every $n$:
			\begin{align*}
				\iim{d_{n+1}^D} \lor (C_n \land \iker{d_n^D}) = (\iim{d_{n+1}^D} \lor C_n) \land \iker{d_n^D}, \\[4pt]
				d_n^{D*} d_{n*}^D C_n = C_n \lor \iker{d_n^D}, \qquad d_{n*}^D d_n^{D*} C_{n-1} = C_{n-1} \land \iim{d_n^D}.
			\end{align*}
	\end{enumerate}
\end{thm}

\section{Arrow categories as homological categories}
\label{sec_arrows}

In Grandis's setup, conditions~\ref{null_exist}--\ref{homologyaxiom} in \Cref{homologicalcatdef} are crucial requirements.
In particular, the homology axiom is relevant for the construction of homology objects by quotienting cycles by boundaries: it guarantees that it does not matter whether one first does the restriction to cycles (by taking a kernel) or the quotient by boundaries (by taking a cokernel).
In our intended application to $\eps$-homological algebra this crucial property does not hold, and this raises the question of how to define homology objects.

The goal of this section is to explain how to get around this problem.
This substantially broadens the scope of homological algebra in general.
For example, the category of groups (\Cref{groups}) is not homological due to the failure of transitivity of normal subgroup inclusions, but the results of this section nevertheless provide a working definition of long exact sequence for any short exact sequence of chain complexes of (non-abelian) groups.

\begin{thm}
	\label{Carrow}
	Let $\cat$ be a category with an ideal of null morphisms $\ideal$ with respect to which kernels and cokernels exist. Then the arrow category $\arrow{\cat}$ is a homological category in the sense of \Cref{homologicalcatdef}, with the null morphisms given by those commuting squares
	\begin{equation}
		\label{square}
		\begin{tikzcd}
			A_0 \ar["f_0"]{r} \ar["a"]{d} & B_0 \ar["b"]{d} \\
			A_1 \ar["f_1"]{r} & B_1
		\end{tikzcd}
	\end{equation}
	whose diagonal $b f_0 = f_1 a$ is a null morphism.

	Furthermore, if $\cat$ satisfies the modularity conditions, then so does $\arrow{\cat}$.
\end{thm}

Here and in the following, we use notation like $a : A_0 \to A_1$ for morphisms that we consider as objects in $\arrow{\cat}$, and we generally draw them vertically.
We also indicate the two components of morphisms in $\arrow{\cat}$ by subscripts $0$ and $1$.

\begin{proof}
	It is clear that the collection of null morphisms in $\arrow{\cat}$ is an ideal since $\ideal$ is. An identity morphism in $\arrow{\cat}$ is null if and only if its underlying object, i.e.~the corresponding morphism in $\cat$, is null. This implies condition~\ref{null_exist}, that a morphism is null if and only if it factors through a null object, since every square~\eqref{square} with null diagonal factors as
	\[
		\begin{tikzcd}
			A_0 \ar["a"]{d} \ar[r,equal] & A_z \ar["f_0"]{r} \ar[d] & B_0 \ar["b"]{d} \\
			A_1 \ar["f_1"]{r} & B_1 \ar[r,equal] & B_1
		\end{tikzcd}
	\]
	with the diagonal of the original square, which is assumed to be in $\ideal$, as the morphism defining the object in the middle.

	We now get to the construction of kernels and cokernels of an arbitrary square~\eqref{square}. The kernel is given by the left square in the diagram
	\[
		\begin{tikzcd}
			\bullet \ar[\monic,"\iker{b f_0}"]{r} \ar{d} & A_0 \ar["f_0"]{r} \ar["a"]{d} & B_0 \ar["b"]{d} \\
			A_1 \ar[equal]{r} & A_1 \ar["f_1"]{r} & B_1
		\end{tikzcd}
	\]
	and the universal property is straightforward to check. Cokernels are constructed dually. 

	As a consequence of \Cref{closed_vs_ker}, the normal monomorphisms in $\arrow{\cat}$ are those squares which are isomorphic to squares of the form
	\[
		\begin{tikzcd}
			A_0 \ar["f_0",\monic]{r} \ar["b f_0"]{d} & B_0 \ar["b"]{d} \\
			B_1 \ar[equal]{r} & B_1
		\end{tikzcd}
	\]
	where $f_0$ is a normal monomorphism in $\cat$ which is $b^* b_*$-closed. Concerning closure under composition, suppose that we compose two normal monomorphisms of this form,
	\[
		\begin{tikzcd}[sep=1cm]
			A_0 \ar["f_0",\monic]{r} \ar["c g_0 f_0" description]{d} & B_0 \ar["cg_0" description]{d} \ar["g_0",\monic]{r} & C_0 \ar["c"]{d} \\
			B_1 \ar[equal]{r} & B_1 \ar[equal]{r} & B_1
		\end{tikzcd}
	\]
	where $f_0$ is $g_0^* c^* c_* g_{0*}$-closed and $g_0$ is $c^* c_*$-closed. In $\Nsb{B_0}$, we have
	\[
		f_0 \leq g_0^* g_{0*} f_0 \leq g_0^* c^* c_* g_{0*} f_0 = f_0,
	\]
	and therefore $f_0$ is $g_0^* g_{0*}$-closed as well. So we record that $f_0$ is the kernel of the morphism $\icoker{g_0 f_0} g_0$.

	Our goal is now to show that the composite $g_0 f_0$ is also normal in $\cat$ by virtue of being the kernel of its cokernel $\icoker{g_0 f_0}$. 
	To this end, consider the diagram
	\[
		\begin{tikzcd}
			& & & \bullet \ar["p"]{dr} \ar[dashed]{dl} \ar[dashed,bend right]{dlll} & & & \bullet \ar[bend left]{dd} \\
			A_0 \ar["f_0",\monic]{rr} \ar["c g_0 f_0"]{dd} & & B_0 \ar["cg_0"]{dd} \ar[\monic,"g_0"]{rr} & & C_0 \ar["c"]{dd} \ar["\icoker{g_0 f_0}" description,\epic]{urr} \ar["\icoker{g_0}" description,\epic]{drr} \\
			& & & & & & \bullet \\
			C_1 \ar[equal]{rr} & & C_1 \ar[equal]{rr} & & C_1
		\end{tikzcd}
	\]
	where $p$ is such that its composition with $\icoker{g_0 f_0}$ is null. Then also its composition with $\icoker{g_0}$ is null, and $p$ factors as indicated by the straight dashed arrow; but since $f_0$ is the kernel of $\icoker{g_0 f_0} g_0$, it follows that $p$ also factors as in the curved dashed arrow. Hence $g_0 f_0$ is the kernel of its cokernel, and therefore a normal monomorphism.

	For the closure of normal monomorphisms under composition, it remains to be shown that $g_0 f_0 \in \Nsb{C_0}$ is $c^* c_*$-closed.
	To see this, we write out the closedness assumptions on $f_0$ and $g_0$ as
	\begin{align}
		\label{cgf_ass}
		\iker{\icoker{cg_0f_0} cg_0} &= f_0, \\
		\label{cg_ass}
		\iker{\icoker{cg_0} c} &= g_0.
	\end{align}
	We then show that $g_0f_0$ has the universal property required of $\iker{\icoker{cg_0f_0}c}$, which proves the claim.
	This follows from the diagram
	\[
		\begin{tikzcd}
			& & & \bullet \ar["p"]{dr} \ar[dashed,"\tilde{p}" description]{dl} \ar[dashed,bend right,"\hat{p}"']{dlll} & & & \\
			A_0 \ar["f_0",\monic]{rr} & & B_0 \ar[\monic,"g_0"]{rr} & & C_0 \ar["c"]{dd} \\
			& & & & & & \\
			& & & & C_1 \ar[\epic,"\icoker{cgf}"]{dr} \ar[\epic,"\icoker{cg}"']{dl} \\
			& & & \bullet \ar{rr} & & \bullet
		\end{tikzcd}
	\]
	since assuming that $\icoker{cg_0f_0} cp$ is null produces the factorization $\tilde{p}$ by~\eqref{cg_ass}, and we then get $\hat{p}$ by~\eqref{cgf_ass}.

	The crucial missing part is the homology axiom~\ref{homologyaxiom}.
	For this, we need to show that given a normal monomorphism composable with a normal epimorphism in $\arrow{C}$ as in the diagram
	\beq
		\label{arrow_homax}
		\begin{tikzcd}[sep=1cm]
			A_0 \ar[\monic,"f_0"]{r} \ar["b f_0"']{d} & B_0 \ar[equal]{r} \ar["b" description]{d} & B_0 \ar["g_1 b"]{d} \\
			B_1 \ar[equal]{r} & B_1 \ar["g_1",\epic]{r} & C_1 
		\end{tikzcd}
	\eeq
	and such that the kernel of $g_1 b$ factors through $f_0$, the composite square can be factored in $\arrow{\cat}$ into a normal epimorphism followed by a normal monomorphism.
	Here, by \Cref{closed_vs_ker} the normality assumptions on the two original squares amount to the equations
	\[
		b^* b_* f_0 = f_0, \qquad b_* b^* \iker{g_1} = \iker{g_1}.
	\]
	We claim that the diagram
	\[
		\begin{tikzcd}[sep=1cm]
			A_0 \ar[equal]{r} \ar["bf_0"']{d} & A_0 \ar["g_1 b f_0" description]{d} \ar["f_0",\monic]{r} & B_0 \ar["g_1b"]{d} \\
			B_1 \ar["g_1",\epic]{r} & C_1 \ar[equal]{r} & C_1 
		\end{tikzcd}
	\]
	provides the desired factorization. In order to see this, it is enough to note that the composite square is clearly the same as in~\eqref{arrow_homax} and to prove the relevant normality conditions, which are
	\[
		b^* g_1^* g_{1*} b_* f_0 = f_0,
	\]
	and the analogous statement for $g_1$ which then follows by duality.
	Using the assumption
	\[
		b^* \iker{g_1} = \iker{g_1 b} \le f_0
	\]
	in $\Nsb{B_0}$, we can conclude that
	\[
		\iker{g_1} = b_* b^* \iker{g_1} \le b_* f_0
	\]
	in $\Nsb{B_1}$ by the monotonicity of $b_*$. Then
	\[
		b^* g_1^* g_{1*} b_* f_0 = b^* (b_* f_0 \lor \iker{g_1}) = b^* b_* f_0 = f_0,
	\]
	where the first equation is by~\eqref{join_formula}.
	This finishes the proof of the main statement.

	Concerning the statement on modularity, we first analyze the construction of pushforwards and pullbacks of normal subobjects.
	So consider any normal subobject that we wish to push forward along an arbitrary morphism,
	\beq
		\label{arrow_pushforward}
		\begin{tikzcd}
			\bullet \ar{d} \ar[\monic,"s"]{r} & A_0 \ar["f_0"]{r} \ar["a"]{d} & B_0 \ar["b"]{d} \\
			A_1 \ar[equal]{r} & A_1 \ar["f_1"]{r} & B_1 
		\end{tikzcd}
	\eeq
	where normality tells us that $a^* a_* s = s$.
	We first need to form the cokernel of the composite morphism. This is given by the right-hand square in
	\[
		\begin{tikzcd}[sep=1.3cm]
			\bullet \ar{d} \ar[\monic,"s"]{r} & A_0 \ar["f_0"]{r} \ar["a"]{d} & B_0 \ar["b"]{d} \ar[equal]{r} & B_0 \ar{d} \\
			A_1 \ar[equal]{r} & A_1 \ar["f_1"]{r} & B_1 \ar[\epic,swap,"\icoker{bf_0s}"]{r} & \bullet
		\end{tikzcd}
	\]
	Hence taking the kernel of this cokernel gives us
	\[
		\begin{tikzcd}[row sep=1.1cm,column sep=2.2cm]
			A_0 \ar[\monic,"\iker{\icoker{bf_0s}b}" shift left=1cm]{r} \ar{d} & B_0 \ar["b"]{d} \ar[equal]{r} & B_0 \ar{d} \\
			B_1 \ar[equal]{r} & B_1 \ar[\epic,swap,"\icoker{bf_0s}"]{r} & \bullet
		\end{tikzcd}
	\]
	where the left-hand square now is the desired pushforward of the original normal monomorphism. The morphism on the upper left is given by
	\[
		\iker{b\hspace{1pt}\icoker{bf_0s}} = b^* \iim{bf_0s} = b^* b_* \iim{f_0 s} = b^* b_* f_{0*} s,
	\]
	which incidentally also makes its normality manifest.
	This finishes the description of pushforwards.
	We may think of this construction this construction as taking the pushforward of $s$ in $\cat$ followed by normalization, which amounts to the application of $b^* b_*$.

	For the construction of pullbacks of normal monomorphisms, we start with
	\[
		\begin{tikzcd}
			& \bullet \ar[\monic,"t"]{dr} \ar{dd} \\
			A_0 \ar["a"]{dd} \ar[near start,crossing over,"f_0"]{rr} & & B_0 \ar["b"]{dd} \\
			& B_1 \ar[equal]{dr} \\
			A_1 \ar["f_1"]{rr} & & B_1
		\end{tikzcd}
	\]
	We first need to take the cokernel in $\arrow{\cat}$, which is given by
	\[
		\begin{tikzcd}[sep=1.3cm]
			A_0 \ar["a"]{d} \ar["f_0"]{r} & B_0 \ar["b"]{d} \ar[equal]{r} & B_0 \ar{d} \\
			A_1 \ar["f_1"]{r} & B_1 \ar[\epic,"\icoker{bt}"]{r} & \bullet
		\end{tikzcd}
	\]
	and then the kernel of the composition. This means that we end up with
	\[
		\begin{tikzcd}[row sep=1.1cm,column sep=2.2cm]
			\bullet \ar["\iker{\icoker{bt}bf_0}",\monic]{r} \ar{d} & A_0 \ar["a"]{d} \ar["f_0"]{r} & B_0 \ar["b"]{d} \ar[equal]{r} & B_0 \ar{d} \\
			A_1 \ar[equal]{r} & A_1 \ar["f_1"]{r} & B_1 \ar[\epic,"\icoker{bt}"]{r} & \bullet
		\end{tikzcd}
	\]
	The resulting morphism top left can be written as
	\[
		\iker{\icoker{bt}bf_0} = f_0^* b^* \iim{bt} = f_0^* b^* b_* t = f_0^* t,
	\]
	which we interpret as stating that pullbacks are constructed just as in $\cat$ itself, and the normality is automatic.

	We then prove that the left modularity equation~\eqref{left_modular} holds.
	Starting with~\eqref{arrow_pushforward} and first taking the pushforward and then the pullback of the normal monomorphism on the left along the square on the right gives us the normal monomorphism with
	\[
		f_0^* b^* b_* f_{0*} s = (b f_0)^* (b f_0)_* s
	\]
	on top.
	If $\cat$ satisfies left modularity for all morphisms, then this can be evaluated to $s \lor \iker{b f_0}$, which proves the left modularity equation in $\arrow{\cat}$.
	For right modularity, we similarly compute
	\begin{align*}
		b^* b_* f_{0*} f_0^* t & = b^* b_* (t \land \iim{f_0}) \\
			& = (t \land \iim{f_0}) \lor \iker{b} \\
			& = t \land (\iim{f_0} \lor \iker{b}) \\
			& = t \land b^* b_* \iim{f_0},
	\end{align*}
	using right and left modularity as well as modularity of the normal subobject lattice in the middle step.
	The latter is based on the fact that $\iker{b} \leq t$, which follows from the normality assumption which guarantees that $t$ is $b^* b_*$-closed.
	Since the image of the $f$ morphism in $\arrow{\cat}$ has $b^* b_* \iim{f_0}$ as the morphism on top, right modularity follows.

	It remains to be shown that every lattice of normal subobjects in $\arrow{\cat}$ is modular. We start with $s_1, s_2, t \in \Nsb{A_0}$ that are $a^* a_*$-closed and satisfy $a_1 \le a_2$, and we prove the modularity equation in the subset of $a^* a_*$-closed elements in $\Nsb{A_0}$.
	By left modularity, the closure operator $a^* a_*$ preserves joins (but not necessarily meets).
	Therefore the desired modularity equation is, expressed in terms of the lattice structure of $\Nsb{A_0}$,
	\begin{align*}
		s_1 \lor a^* a_* (t \land s_2) & = (s_1 \lor \iker{a}) \lor (t \land s_2) \\
			& = (s_1 \lor \iker{a} \lor t) \land s_2 \\
			& = \iker{a} \lor ((s_1 \lor t) \land s_2) \\
			& = a^* a_*((s_1 \lor t) \land s_2),
	\end{align*}
	where the application of modularity in the second and third steps also uses $\iker{a} \le s_2$, which is a consequence of the $a^* a_*$-closedness of $s_2$.
	This finally completes the proof.
\end{proof}

We separately record an important observation made in the proof.

\begin{prop}
	For every object $a : A_0 \to A_1$ in $\arrow{\cat}$, its lattice of normal subobjects has natural identifications
	\[
		\Nsb{a} = \{ s \in \Nsb{A_0} \mid a^* a_* s = s\} = \{t \in \Nsb{A_1} \mid a_* a^* t = t \}.
	\]
\end{prop}

Let us now get to homology. We can use the canonical inclusion functor $\cat\to\arrow{\cat}$, which maps every object to its identity morphism, in order to treat the homology of chain complexes in $\cat$ in the more well-behaved category $\arrow{\cat}$. Here is the main idea:

\begin{prop}
	Given two composable squares
	\[
		\begin{tikzcd}
			A_0 \ar["f_0"]{r} \ar["a"]{d} & B_0 \ar["b"]{d} \ar["g_0"]{r} & C_0 \ar["c"]{d} \\
			A_1 \ar["f_1"]{r} & B_1 \ar["g_1"]{r} & C_1
		\end{tikzcd}
	\]
	with null composite in $\arrow{\cat}$, their homology object is given by
	\beq
		\label{arrow_homology_gen}
		\begin{tikzcd}[sep=1cm]
			\iKer{g_1 b} \ar["\icoker{bf_0} f_0 \iker{g_1 b}" description]{d} \\
			\iCoker{bf_0}
		\end{tikzcd}
	\eeq
\end{prop}

The proof follows the verification of the homology axiom in the proof of \Cref{Carrow}, so we omit further details.

Note that the homology object only depends on the composable triple $(f_0,b,g_1)$. In the case where the sequence comes from $\cat$, meaning that its three $\arrow{\cat}$-objects are identities,
\[
	\begin{tikzcd}
		A \ar["f"]{r} \ar[equal]{d} & B \ar[equal]{d} \ar["g"]{r} & C \ar[equal]{d} \\
		A \ar["f"]{r} & B \ar["g"]{r} & C
	\end{tikzcd}
\]
the homology object specializes to
\beq
	\label{arrow_homology}
	\begin{tikzcd}[sep=1cm]
		\iKer{g} \ar["\icoker{f} \iker{g}" description]{d} \\
		\iCoker{f}
	\end{tikzcd}
\eeq
We think of this as the canonical map from cycles to chains modulo boundaries.
Working in $\arrow{\cat}$ instead of $\cat$ hence amounts to keeping track of this information more explicitly.

Upon combining the long exact sequences of Grandis's \Cref{longexact} with our \Cref{Carrow}, we get the following version of long exact sequences. 

\begin{cor}
	\label{arrowlongexact}
	Let $\cat$ be a category with ideal of null morphisms $\ideal$ with respect to which kernels and cokernels exist. Let
	\[
		\begin{tikzcd}
			(C_\bullet,d_\bullet^C) \ar["i_\bullet"]{r} & (D_\bullet,d_\bullet^D) \ar["q_\bullet"]{r} & (E_\bullet,d_\bullet^E)
		\end{tikzcd}
	\]
	be a short exact sequence of chain complexes in $\cat$.
	Then there is an induced long sequence of homology objects in the arrow category, which in $\cat$ takes the form
	\beq
		\label{arrow_les}
		\begin{tikzcd}[row sep=1.4cm,column sep=.7cm]
			\ldots \ar{r} & \iKer{d_{n+1}^E q_{n+1}} \ar{d} \ar["\partial_{n+1,0}"]{r} & \iKer{d_n^D i_n} \ar{d} \ar{r} & \iKer{d_n^D} \ar{d} \ar{r} & \iKer{d_n^E q_n} \ar{d} \ar["\partial_{n,0}"]{r} & \iKer{d_{n-1}^D i_{n-1}} \ar{d} \ar{r} & \ldots \\
			\ldots \ar{r} & \iCoker{q_n d_{n+2}^D}          \ar["\partial_{n+1,1}"]{r} & \iCoker{i_n d_{n+1}^C}  \ar{r} & \iCoker{d_{n-1}^D}  \ar{r} & \iCoker{q_n d_{n+1}^D}  \ar["\partial_{n,1}"]{r} & \iCoker{i_{n-1} d_n^C}      \ar{r} & \ldots
		\end{tikzcd}
	\eeq
	\begin{enumerate}
		\item The unlabelled arrows are the obvious ones.
		\item The connecting morphisms $\partial_{n,\ast}$ are natural in the short exact sequence.
		\item The composite diagonal of every two subsequent square is null.
		\item The sequence is exact if the following instances of the modularity condition hold in $\Nsb{D_n}$ for every $n$:
			\begin{align*}
				\iim{d_{n+1}^D} \lor (C_n \land \iker{d_n^D}) = (\iim{d_{n+1}^D} \lor C_n) \land \iker{d_n^D}, \\[4pt]
				d_n^{D*} d_{n*}^D C_n = C_n \lor \iker{d_n^D}, \qquad d_{n*}^D d_n^{D*} C_{n-1} = C_{n-1} \land \iim{d_n^D}.
			\end{align*}
	\end{enumerate}
\end{cor}

Since the assumptions are weaker, this result has much greater generality than Grandis's, at the cost of the additional complication of a larger diagram.
It should also be noted that in this diagram the only homology objects in the sense of~\eqref{arrow_homology} are those of the chain complex $D$; the other vertical arrows are not of this kind.

\begin{proof}
	The original short exact sequences
	\[
		\begin{tikzcd}
			C_n \ar[\monic,"i_n"]{r} & D_n \ar[\epic,"q_n"]{r} & E_n
		\end{tikzcd}
	\]
	in $\cat$ induce short exact sequences in $\arrow{\cat}$ of the form
	\[
		\begin{tikzcd}
			C_n \ar[\monic,"i_n"]{d} \ar[\monic,"i_n"]{r} & D_n \ar[equal]{r} \ar[equal]{d} & D_n \ar[\epic,"q_n"]{d} \\
			D_n \ar[equal]{r} & D_n \ar[\epic,"q_n"]{r} & E_n
		\end{tikzcd}
	\]
	We therefore obtain a short exact sequence of chain complexes in $\arrow{\cat}$, and \Cref{longexact} applies thanks to \Cref{Carrow}.
	Using~\eqref{arrow_homology_gen} and~\eqref{arrow_homology} for the computation of the homology objects then produces the vertical arrows in~\eqref{arrow_les}.
	For example for the second vertical arrow, we start with
	\[
		\begin{tikzcd}
			C_{n+1} \ar[\monic,"i_{n+1}"]{d} \ar["d_{n+1}^C"]{r} & C_n \ar[\monic,"i_n"]{d} \ar["d_n^C"]{r} & C_{n-1} \ar[\monic,"i_{n-1}"]{d} \\
			D_{n+1} \ar["d_{n+1}^D"]{r} & D_n \ar["d_n^D"]{r} & D_{n-1}
		\end{tikzcd}
	\]
	and apply~\eqref{arrow_homology_gen}.

	Now the given modularity conditions translate exactly into those of \Cref{longexact}.
\end{proof}

\begin{rem}
	Long exact sequences in the form~\eqref{arrow_les} may even be of interest in classical cases, say when $\cat$ is an abelian category, due to the fact that these homology objects contain some extra information by keeping track of cycles and the quotient by boundaries separately.
	On the other hand, as far as we can see their exactness property does not provide additional information beyond the exactness of the standard long exact sequences.
\end{rem}

\section{Homological algebra up to \texorpdfstring{$\eps$}{ε}?}
\label{seceps}

We now apply the above results in the context of normed spaces in order to sketch an approach to homological algebra for ``$\eps$-curved complexes'', by which we mean chain complexes of normed spaces where the differential only squares to some map of norm at most $\eps$. We fix some real number $\eps \in (0,1)$ throughout as a parameter for the theory.

For technical convenience---the reason for which will become clear in the construction of $\eps$-cokernels at~\eqref{epscokernel}---we work with \newterm{seminorms} and \newterm{seminormed spaces}. These are defined just like norms and normed spaces except in that the non-degeneracy condition in the definition of norm is dropped~\cite[Section~II.1]{schaefer}.
We assume $\R$ as the base field throughout.\footnote{We expect that all the definitions and results of this section can be adapted straightforwardly to the complex case, but we have not done so explicitly.}

\begin{defn}
	\label{normcatdef}
	$\Norm$ is the category with seninormed spaces as objects, and morphisms $(V,\|\cdot\|)\to (W,\|\cdot\|)$ equivalence classes of linear maps $f : V \to W$ of operator norm $\|f\| \leq 1$, where the equivalence relation is defined as $f\sim g$ if and only if
	\[
		\| f(x) - g(x) \| = 0 \qquad \forall x\in V.
	\]
\end{defn}

It is understood that composition in $\Norm$ is given by ordinary composition of representing maps, and it is easy to see that this is well-defined.
We abuse terminology and notation by leaving the distinction between morphisms and their representing linear maps implicit.

\begin{lem}
	\label{seminorm_norm}
	Every seminormed space $(V,\|\cdot\|)$ is isomorphic in $\Norm$ to the normed space $(V / V_0,{\|\cdot\|'})$, where
	\[
		V_0 \coloneqq \{ x \in V \mid \|x\| = 0 \}
	\]
	is the null space and $\|\cdot\|'$ is the induced quotient norm given by $\|\cdot\|$ on representatives.
\end{lem}

\begin{proof}
	The triangle inequality immediately shows that $\|\cdot\|'$ is well-defined.
	We show that the canonical projection $q : V \to V / V_0$ is an isomorphism. It satisfies $\|q\| \le 1$ by definition. Since every surjective linear map has a linear section, we can find a right inverse $j : V/V_0 \to V$ for $q$, and it is again clear that $\|j\| \le 1$.
	While $qj$ is equal to the identity map on the nose, the composite $qj$ also represents the identity morphism since $x - q(j(x))$ is null for every $x \in V$.
\end{proof}
		
\begin{prop}	
	$\Norm$ is equivalent to the category of normed spaces and linear maps of operator norm $\le 1$.
\end{prop}

\begin{proof}
	The inclusion functor from the standard category of normed spaces into $\Norm$ is clearly fully faithful.
	It is essentially surjective by \Cref{seminorm_norm}.
\end{proof}

\begin{lem}
	\label{norm_iso}
	A linear map $f : V \to W$ defines an isomorphism in $\Norm$ if and only if the following hold:
	\begin{enumerate}
		\item $\|f(x)\| = \|x\|$ for all $x \in V$.
		\item $f$ is \newterm{essentially surjective}: for every $y \in W$ there is $x \in V$ with $\| y - f(x) \| = 0$.
	\end{enumerate}
\end{lem}

\begin{proof}
	If an inverse exists, as represented by a map $g : W \to V$ with $\|g\| \le 1$, then it is straightforward to see that these conditions hold.
	Conversely, if these conditions hold, then $f$ clearly induces an isomorphism between $V / V_0$ and $W / W_0$, so that the claim follows by the previous lemma.
\end{proof}

We now turn to the definition of null morphisms that lies behind our approach to $\eps$-homological algebra, and prove that these null morphisms form an ideal to which our general results apply.
The construction of kernels and cokernels given in the proof below is central to how we think of $\eps$-homological algebra.

\begin{prop}
	\label{Normideal}
	In $\Norm$, the morphisms of norm $\leq \eps$ form an ideal $\ideal_\eps$ with respect to which kernels and cokernels exist.
\end{prop}

From now on, we will speak of \newterm{$\eps$-kernels} and \newterm{$\eps$-cokernels} in order to emphasize the dependence on the parameter $\eps$.\footnote{Recall our assumption $\eps \in (0,1)$. Although $\eps = 1$ could be considered, the theory is not of interest in this case since this choice makes every morphism null. And setting $\eps = 0$ in order to recover ordinary homological algebra with vector spaces makes sense up to the definition of $\ideal_0$, but not beyond that since $\eps^{-1}$ appears already in the proof of \Cref{Normideal}.}

\begin{proof}
	The basic inequality $\|gf\| \leq \|g\| \, \|f\|$ for any two composable morphisms $f$ and $g$ shows that if $f$ or $g$ has norm $\leq \eps$, then so does $gf$, given that they both have norm $\leq 1$ in any case. Hence $\ideal_\eps$ is indeed an ideal.

	Now let $f : V\to W$ be a morphism in $\Norm$. We claim that its kernel is given by the vector space $V$ itself equipped with the modified \newterm{$\eps$-kernel seminorm},
	\begin{equation}
		\label{epskernel}
		\|x\|_{\iker{f}_\eps} \coloneqq \max \left( \|x\|, \eps^{-1} \|f(x)\| \right).
	\end{equation}
	Since $\|x\| \leq \|x\|_{\iker{f}_\eps}$ holds by definition, we indeed have a well-defined inclusion morphism
	\[
		\iker{f}_\eps : (V,\|\cdot\|_{\iker{f}_\eps}) \longrightarrow (V,\|\cdot\|).
	\]
	The inequality $\|f(x)\| \le \eps \|x\|_{\iker{f}_\eps}$ shows that its composition with $f$ is in $\ideal_\eps$.
	To check the universal property, suppose that we have $g : U \to V$ such that $\|fg\| \leq \eps$. Since $\iker{f}_\eps$ is a monomorphism by construction, it is enough to show that $g$ factors through $\iker{f}_\eps$, or equivalently that
	\[
		\|g(u)\|_{\iker{f}_\eps} \leq \|u\|
	\]
	for all $u\in U$. By the definition of the left-hand side as a maximum, this is equivalent to the conjunction of $\|g(u)\| \leq \|u\|$ and $\eps^{-1}\|f(g(u))\| \leq \|u\|$. While the former holds thanks to $\|g\| \leq 1$, the latter is a restatement of the assumed $\|fg\| \leq \eps$.

	We now get to the construction of $\eps$-cokernels, for the same $f : V \to W$. We take the $\eps$-cokernel to be given by $W$ itself, now equipped with the modified seminorm\footnote{Here we can see why it is convenient to work with seminorms rather than norms. If we did the latter, then we would now have to take the quotient by the null space explicitly, which would complicate the exposition.}
	\begin{equation}
		\label{epscokernel}
		\|y\|_{\icoker{f}_\eps} \coloneqq \inf_{x\in V} \left( \|y - f(x) \| + \eps \|x\| \right).
	\end{equation}
	for all $y\in W$. Similarly to the $\eps$-kernel case, since $\|y\|_{\icoker{f}_\eps} \leq \|y\|$ holds by construction, the identity map on $W$ induces a well-defined morphism
	\[
		\icoker{f}_\eps : (W,\|\cdot\|)\to (W,\|\cdot\|_{\icoker{f}_\eps}).
	\]
	Because of
	\[
		\|f(x)\|_{\icoker{f}_\eps} \leq \eps \|x\|
	\]
	for all $x\in V$, the composition of $f$ with $\icoker{f}_\eps$ indeed has norm at most $\eps$.\footnote{However, it is not difficult to find examples where the above inequality is strict, e.g.~taking $f = 0$.} To check the universal property, let $h : W \to X$ be another morphism, and suppose that $\|hf\| \leq \eps$.
	Then what we need to show for the unique factorization is that
	\[
		\|h(y)\| \leq \|y\|_{\icoker{f}_\eps}
	\]
	for all $y\in W$. To this end, it is enough to prove that
	\[
		\|h(y)\| \leq \|y - f(x)\| + \eps \|x\|
	\]
	for all $x\in V$ and $y\in W$. But this follows from
	\[
		\|h(y)\| \leq \|h(y) - h(f(x))\| + \|h(f(x))\| \leq \|y - f(x)\| + \eps \|x\|,
	\]
	using $\|h\|\leq 1$ and $\|hf\| \leq \eps$.
\end{proof}

\begin{rem}
	Geometrically, these constructions of kernel and cokernel have the following meaning. The unit ball of the $\eps$-kernel norm~\eqref{epskernel} is given by
	\[
		V_1 \cap \eps f^{-1}(W_1).
	\]
	The unit ball of the $\eps$-cokernel norm~\eqref{epscokernel} is given by the smallest unit ball which contains both $W_1$ and $\eps^{-1} f(V_1)$, which is
	\[
		\conv{\eps^{-1} f(V_1) \cup W_1}
	\]
	together with all additional $y\in W$ for which $ty$ lies in this convex hull for every $t \in [0,1)$.
	It may help the reader to keep these geometrical descriptions in mind, identifying every seminorm with the gauge of its unit ball~\cite[II.1.4]{schaefer}.
\end{rem}

We note a few additional consequences of the construction of $\eps$-kernels and $\eps$-cokernels given in the proof.

\begin{rem}
	\begin{enumerate}
		\item For $f : V \to W$, its image $\iIm{f}$ is given by $W$ itself equipped with the \newterm{$\eps$-image seminorm}
			\beq
				\label{eps_im}
				\|y\|_{\iim{f}_\eps} \coloneqq \max \left(\|y\|,\inf_{x \in V} (\eps^{-1} \|y - f(x)\| + \|x\| ) \right).
			\eeq
			As one might expect, this diverges as $\eps \to 0$ for every $y$ that is not in the closure of the set-theoretic image of $f$.
		\item The coimage of $f$ is given by $V$ itself with the \newterm{$\eps$-coimage seminorm}
			\beq
				\label{eps_coim}
				\|x\|_{\icoim{f}_\eps} \coloneqq \inf_{x' \in V} \bigg( \|x - x'\| + \max( \eps \|x'\|, \|f(x')\| ) \bigg) = \max(\eps \|x\|, \|f(x)\|).
			\eeq
			We prove the second equation by arguing that the infimum is achieved at $x' = x$. Indeed for general $x'$, we get by $\eps \le 1$,
			\[
				\|x - x'\| + \eps \|x'\| \ge \eps(\|x\| - \|x'\|) + \eps \|x'\| = \eps \|x\|,
			\]
			which proves the claim if $\|f(x)\| \le \eps \|x\|$. Similarly,
			\[
				\|x - x'\| + \|f(x')\| \ge \|f(x)\| - \|f(x')\| + \|f(x')\| = \|f(x)\|,
			\]
			which applies in case that $\|f(x)\| \ge \eps \|x\|$.

			In particular, for $\eps \to 0$, the coimage seminorm degenerates to $\|x\|_{\icoim{f}_0} = \|f(x)\|$, which is again what one might expect.
	\end{enumerate}
\end{rem}

The following result seems rather odd.

\begin{prop}
	\label{normal_chars}
	For a general morphism $f : V \to W$ in $\Norm$, the following are equivalent:
	\begin{enumerate}
		\item\label{eps_mono} $f$ is a normal monomorphism.
		\item\label{eps_epi} $f$ is a normal epimorphism.
		\item\label{eps_balanced} $f$ is essentially surjective and
			\beq
				\label{lower_bound}
				\|x\| \le \eps^{-1} \|f(x)\| \quad \forall x \in V.
			\eeq
	\end{enumerate}
\end{prop}

\begin{proof}
	We show first that \ref{eps_mono} is equivalent to~\ref{eps_balanced}.
	In the forward direction, it is enough to show that \ref{eps_balanced} holds for all morphisms of the form $f = \iker{g}$, where $g$ is arbitrary.
	Essential surjectivity is clear by the construction of $\eps$-kernels given above.
	Moreover, the formula~\eqref{epskernel} gives
	\[
		\|x\|_{\iker{g}_\eps} = \max(\|x\|, \eps^{-1} \|g(x)\|) \le \eps^{-1} \|x\|,
	\]
	which is true by $\eps \le 1$ and $\|g\| \le 1$ and corresponds to~\eqref{lower_bound}.

	Conversely, if $f$ satisfies~\ref{eps_balanced}, then we prove that $f$ is the kernel of its cokernel by showing that the induced morphism $V \to \iIm{f}$ is an isomorphism.
	By \Cref{norm_iso}, the above description of the image and the essential surjectivity assumption, this amounts to showing that for all $x \in V$,
	\[
		\max \left(\|f(x)\|,\inf_{x' \in V} (\eps^{-1} \|f(x) - f(x')\| + \|x'\| ) \right) = \|x\|.
	\]
	This follows by arguing that the infimum is attained at $x' = x$.
	Indeed,
	\[
		\eps^{-1} \|f(x) - f(x')\| + \|x'\| \ge \|x - x'\| + \|x'\| \ge \|x\|,
	\]
	as was to be shown.

	We now turn to the equivalence of~\ref{eps_epi} and~\ref{eps_balanced}.
	In the forward direction, it is enough to show that \ref{eps_balanced} holds for all morphisms of the form $f = \icoker{h}$, where $h$ is arbitrary.
	Indeed the formula~\eqref{epscokernel} gives
	\begin{align*}
		\|y\| & = \inf_{x \in \dom{h}} \left( \|y - h(x)\| + \|h(x)\| \right) \\
			& \le \inf_{x \in \dom{h}} \left( \eps^{-1} \|y - h(x) \| + \|x\| \right) \\
			& = \eps^{-1} \|y\|_{\icoker{f}_\eps}
	\end{align*}
	for any $y$ in the codomain of $h$, as was to be shown for~\eqref{lower_bound}.

	Conversely, if $f$ satisfies~\ref{eps_balanced}, then we prove that $f$ is the cokernel of its cokernel by showing that the induced morphism $\iCoim{f} \to W$ is an isomorphism.
	By \Cref{norm_iso}, the above description of the coimage and the essential surjectivity assumption, this amounts to showing that for all $x \in V$,
	\[
		\max \left( \eps \|x\|, \|f(x)\| \right) = \|f(x)\|.
	\]
	This holds because it is obviously equivalent to the assumed~\eqref{lower_bound}.
\end{proof}

We briefly consider the lattice of normal subobjects.

\begin{cor}
	\label{nsb_norm}
	For any seminormed space $V$, its lattice of normal subobjects with respect to $\ideal_\eps$ has a natural identification with
	\[
		\Nsb{V} \cong \{ \textrm{seminorms } |\cdot| \textrm{ on } V \mid \|x\| \le |x| \le \eps^{-1} \|x\| \}.
	\]
\end{cor}

Of course, the partial order on $\Nsb{V}$ corresponds to the reverse order between these seminorms, or equivalently (and perhaps more intuitively) to forward inclusion of unit balls.
By the correspondence between normal subobjects and normal quotients of~\eqref{nqt_nsb}, the lattice of normal quotients $\Nqt{V}$ has the same form.

\begin{proof}
	By \Cref{norm_iso} and the essential surjectivity of normal monomorphisms, we can assume without loss of generality that the underlying vector space of every normal subobject is $V$ itself, with inclusion map being the identity.
	But then the claim follows by \Cref{normal_chars}.
\end{proof}

\begin{rem}
	The lattice structure on $\Nsb{V,\|\cdot\|}$ is given by
	\begin{align*}
		%\label{norm_lattice}
		(|\cdot|_1 \land |\cdot|_2)(x) & = \max( |x|_1, |x|_2 ), \\
		(|\cdot|_1 \lor |\cdot|_2)(x) & = \inf_{x_1, x_2 \in V \: : \: x_1 + x_2 = x} (|x_1|_1 + |x_2|_2).
	\end{align*}
	which is the restriction of the lattice structure on the set of \emph{all} seminorms on the vector space $V$.
	Unfortunately, $\Nsb{V}$ is typically\footnote{Although we only present one single counterexample to modularity in what follows, its structure indicates that this counterexample is rather generic. More precisely, we conjecture that $\Nsb{V,\|\cdot\|}$ is modular if and only if $\dim(V / V_0) \le 1$.} not modular.
	For example on $\R^2$, consider the following scalar multiples of the standard $p$-norms,
	\[
		\frac{1}{4} \|\cdot\|_1, \quad \frac{1}{3} \|\cdot\|_\infty, \quad \frac{1}{5} \|\cdot\|_1.		
	\]
	Per \Cref{nsb_norm}, these are in the normal subobject lattice on $\R^2$ with respect to any sufficiently small reference norm and any sufficiently small $\eps$.
	Then, considering the polyhedral unit balls and their vertices as in \Cref{unit_balls_R2}, some calculation shows that
	\begin{align*}
		\left( \frac{1}{4} \|\cdot\|_1 \lor \left( \frac{1}{3} \|\cdot\|_\infty \land \frac{1}{5} \|\cdot\|_1 \right) \right) \left( \frac{7}{2}, \frac{3}{2} \right) & = \frac{17}{16}, \\
		\left( \left( \frac{1}{4} \|\cdot\|_1 \lor \frac{1}{3} \|\cdot\|_\infty \right) \land \frac{1}{5} \|\cdot\|_1 \right) \left( \frac{7}{2}, \frac{3}{2} \right) & = 1.
	\end{align*}
	Modularity fails since these results are not equal.
\end{rem}

\begin{figure}
	\centering
	\begin{tikzpicture}[scale=0.65]
		\draw[->,very thick] (-6,0)--(6,0) node[right]{};
		\draw[->,very thick] (0,-6)--(0,6) node[above]{};
		\draw[red] (5.0,-3.0) -- (2.33,5.0);
		\draw[fill=gray,fill opacity=0.2] (3,3) -- (3,-3) -- (-3,-3) -- (-3,3) -- cycle;
		\draw[fill=gray,fill opacity=0.2] (4,0) -- (0,4) -- (-4,0) -- (0,-4) -- cycle;
		\draw[fill=gray,fill opacity=0.2] (5,0) -- (0,5) -- (-5,0) -- (0,-5) -- cycle;
		\draw[ForestGreen] (3,2) -- (4,0);
		\draw[ForestGreen] (0,0) -- (3.5,1.5);
		\node[circle,fill,draw,blue,inner sep=1pt,label={[blue,above right]$(\frac{7}{2},\frac{3}{2})$}] at (3.5,1.5) {};
	\end{tikzpicture}
	\caption{}
	\label{unit_balls_R2}
\end{figure}

Our final goal is to sketch some basics of $\eps$-curved homological algebra in our framework. 
Our ideal of null morphisms $\ideal_\eps$ has been chosen precisely such that chain complexes in $(\Norm,\ideal_\eps)$ amount to the following.

\begin{defn}
	\label{epscurvedcomplex}
	An \newterm{$\eps$-curved chain complex} $(C_\bullet,d_\bullet)$ is a sequence of seminormed spaces $(C_n)_{n\in\Z}$ together with morphisms $d_n : C_n \to C_{n-1}$ such that
	\[
		\|d_n \| \leq 1, \qquad \| d_{n+1} d_n \| \leq \eps \qquad \forall n \in \Z.
	\]
\end{defn}

The condition $\|d_n\| \leq 1$ is technically redundant since every morphism has norm $\le 1$ by definition, but we nevertheless restate it for emphasis. The main point of this definition is that the differential does not need to square to zero, since chains of the form $d_{n-1} d_n c$ are merely required to be small in norm, $\|d_{n-1} d_n c\| \le \eps \|c\|$.

\begin{prop}
	\label{eps_exact_char}
	An $\eps$-curved chain complex $C$ is exact if and only if for every $n \in \Z$ and $y \in C_n$,
	\beq
		\label{eps_exact}
		\inf_{x \in C_{n+1}} \left( \| y - d_{n+1} x \| + \eps \|x\| \right) \le \max \left( \eps \|y\|, \|d_n y\| \right).
	\eeq
\end{prop}

\begin{proof}
	In light of $\iim{d_{n+1}} = \iker{d_n}$ being the definition of exactness, the formulas for kernel and image given as \eqref{epskernel} and~\eqref{eps_im} show that exactness holds if and only if
	\[
		\max \left( \|y\|, \inf_{x \in C_{n+1}} \left( \eps^{-1} \| y - d_{n+1} x \| + \|x\| \right) \right) = \max \left( \|y\|, \eps^{-1} \|d_n y\| \right).
	\]
	for all $y$.
	But since the inequality direction $\ge$ is clear by the assumed $\iim{d_{n+1}} \le \iker{d_n}$,\footnote{It is an instructive exercise to derive the inequality $\ge$ directly from $\|d_n d_{n+1}\| \le \eps$.} this is manifestly equivalent to the inequality in the statement.
\end{proof}

The following observation suggests that our version of $\eps$-curved homological algebra has the potential to provide a general framework for Kazhdan's ad-hoc arguments.

\begin{rem}
	\label{kazhdan}
	Our $\eps$-exactness condition~\eqref{eps_exact} is closely related to Kazhdan's notion of \emph{$\eps$-acyclicity}, which is defined in~\cite[p.~316]{kazhdan} as the existence of $x \in C_{n+1}$ for every $y \in C_n$ such that
	\[
		\|x\| \le \|y\|, \qquad \|y - d_{n+1} x \| \le \eps \|y\| + \| d_n y \|.
	\]
	Indeed up to a constant factor\footnote{Note also that the differentials used by Kazhdan only satisfy $\|d_n\| \le n+1$ rather than $\|d_n\| \le 1$. Rescaling them to our convention will introduce another scalar factor beyond what we discuss here.}, we can write our condition~\eqref{eps_exact} as equivalent to the existence of $x$ such that
	\[
		\|x\| \le \|y\| + \eps^{-1} \|d_n y\|, \qquad \|y - d_{n+1} x \| \le \eps \|y\| + \|d_n y\|,
	\]
	which is more permissive than Kazhdan's definition due to the extra term in the first inequality, but otherwise coincides with it.
	Going through Kazhdan's use of $\eps$-acyclicity~\cite[p.~319]{kazhdan} shows that his argument still works even with this more permissive notion of $\eps$-exactness (with a slightly worse bound on the constants). 
\end{rem}

To conclude, we believe to have demonstrated that our framework, based on Grandis's approach to non-abelian homological algebra, is a good candidate for a theory of $\eps$-curved homological algebra.
We expect that the application of the results of \Cref{sec_arrows} to $\Norm$ with $\ideal_\eps$ as its ideal of null morphisms can take this further, to the point where a well-behaved $\eps$-version of group cohomology and some tools for its calculation can be developed, involving homology objects of the form~\eqref{arrow_homology_gen} and~\eqref{arrow_homology}.
We leave further investigations of this idea to future work.

\bibliographystyle{plain}
\bibliography{approximate_cohomology}

\end{document}